\documentclass {article}

\usepackage{amsfonts,amsmath,latexsym}
\usepackage{amstext,amssymb}
\usepackage{amsthm}
\usepackage {epsf}
\usepackage {epsfig}
\usepackage[scanall]{psfrag}
\usepackage {graphicx}
\usepackage{verbatim}

\newcommand {\R}{\mathbb{R}}
\newcommand {\Ss}{\mathbb{S}}

\newcommand {\calS}{\mathcal{S}}
\newcommand {\calT}{\mathcal{T}}

\newcommand {\calI}{\mathcal{I}}

\newcommand {\calW}{\mathcal{W}}

\newcommand {\calF}{\mathcal{F}}
\newcommand {\calM}{\mathcal{M}}

\DeclareMathOperator{\Atn}{Atn}

\newtheorem {thm} {Theorem}

\newtheorem {lemma} [thm] {Lemma}
\newtheorem {cor} [thm] {Corollary}

\begin{document}
\title{
\bf  A Payne-Weinberger eigenvalue estimate for wedge domains on spheres
}
\author{  Jesse Ratzkin \\
University of Georgia \\
{\tt jratzkin@math.uga.edu}
\and
Andrejs Treibergs \\University of Utah\\{\tt treiberg@math.utah.edu}}

\date{\today}

\maketitle

\begin{abstract}
A Faber-Krahn type argument gives a sharp lower estimate for the first 
Dirichlet eigenvalue for subdomains of  wedge domains in spheres, 
generalizing the inequality for the plane, found by Payne and Weinberger.
An application is an alternative proof to the finiteness of a Brownian 
motion capture time estimate. \end{abstract}

Many  lower estimates for the first Dirichlet eigenvalue of a domain stem 
from an inequality between a line integral and an area 
integral \cite[pp. 85--133]{Ch}, \cite[pp. 37--40]{LT}, 
\cite[pp. 462--467]{P}. These inequalities are often sharp, in that 
equality of the eigenvalues implies a geometric equality. 
For example, the Faber-Krahn inequality 
\cite{F},~\cite{K},  proved by comparing level sets of the 
eigenfunction using the classical isoperimetric inequality, reduces 
to equality for round disks.  Cheeger's inequality \cite{C} bounds the 
eigenvalue from below in terms of the minimal ratio of area to length of 
subdomains. 

Our main result, Theorem~\ref{Main_th}, is a lower bound for the first 
Dirichlet 
eigenvalue for a domain contained in a wedge in a two sphere, generalizing
an eigenvalue estimate of Payne and Weinberger~\cite{PW}, \cite[p.462]{P}
for planar domains contained in a wedge. As 
an application, we give an alternative proof of our Brownian capture 
time estimate~\cite{RT}. Curiously, our proof does not seem to carry over to 
domains contained in a wedge in the hyperbolic plane. 

If $(\rho, \theta)$ are polar coordinates centered at a pole of 
$\Ss^2$, recall that the round metric is given by 
$$ds^2 = d\rho^2 + \sin^2 \! \rho \, d\theta^2.$$
Let $\calW=\{(\rho,\theta): 0\le \theta\le\pi/\alpha,\  0\le 
\rho<\pi\}$ be the sector in $\Ss^2$ of angle $\pi/\alpha$, for 
$\alpha > 1$, and let $G$ be a domain such that $\overline{G} \subset 
\calW$ is compact. Also define the truncated sector~$\calS(r) := 
\{ (\rho, \theta): 0 \leq \theta \leq \pi/\alpha, 0 \leq \rho \leq r\}$. 
Observe that 
\begin{equation}\label{w_eqn}
w = \tan^\alpha\!\!\left(\frac {\rho}{2}\right) \, \sin \alpha \theta
\end{equation}
is a positive harmonic function in $\calW$, with zero boundary values. 
\begin{thm}\label{Main_th} For every subdomain $G$ with compact 
$\overline{G}\subset \calW$, we have the estimate 
\begin{equation}\label{EV_eq}
\lambda_1(G)\ge \lambda_1(\calS(r^*)),
\end{equation}
where $r^*$ is chosen such that
$$
\calI(G)=\int_G w^2\, da = \int_{\calS(r^*)} w^2\, da.
$$
Equality holds if and only if $G$ is the sector~$\calS(r^*)$.
\end{thm}

Our argument is similar to the proof of the planar version 
in~\cite{PW}. Our main tool is an isoperimetric-type inequality, 
Lemma \ref{isop_lemma}, which we prove in Section \ref{isop_sec}. 
We use this inequality to estimate the Rayleigh quotient of a test 
function, proving Theorem \ref{Main_th}, in Section~\ref{rayliegh_sec}. 
Finally, in Section \ref{appl_sec}, we apply our eigenvalue 
estimate to a problem in Brownian pursuit. 

\section{Isoperimetric Inequality} \label {isop_sec}

In this section we prove an isoperimetric inequality for moments 
of inertia of a domain $G \subset \calW$. Later we will use this 
inequality to estimate the Raleigh quotient of admissible 
functions in $G$. 

We begin by stating a version Szeg\H o's Lemma \cite {Sz}:
\begin {lemma} \label{Szego_lemma} 
Let $\psi,\phi:[0,\omega)\to[0,\infty)$ be locally integrable functions 
with $\psi$ nonnegative and $\phi$ nondecreasing. Let $\Phi(y)=\int_0^y 
\phi(t)\, dt$ and $\Psi(x)=\int_0^y \psi(s)\, ds$ be their primitives. 
Let $E\subset [0,\omega)$ be a bounded measurable set. Then
\begin{equation}\label{eqn_Szego}
 \Phi\left( \int_E \psi(x)\, dx\right)\le  \int_E \phi(\Psi(x))\, \psi(x)\, dx.
\end{equation}
For $\phi$ increasing, equality holds if and only if the measure of 
$E \cap [0,R]$ is $R$. 
\end {lemma}

\begin{proof} Let $\mu$ be Lesbesgue measure with line element 
$dx$ and define the measure $\nu$ by $d\nu = \psi\, dx$. Then 
$\nu$ is absolutely continuous with respect to $\mu$ and, using 
the Radon--Nikodym Theorem, when we change variables $y = \Psi(x)$
we have $dy = \psi(x) dx$. Let $E'$ be the image of $E$ under 
the map $\Psi$, with characteristic function $\chi_{E'}$, so that 
$\Phi(\int_{E'} dy) = \Phi (\int_E \psi(x) dx)$. Next, because 
$\phi$ is nondecreasing, for $y\ge 0$,
$$
\phi \left ( \int_0^y \chi_{E'} dy\right ) \leq \phi(y).
$$
Moreover, for $\phi$ increasing, equality holds if and only if 
$\mu(E'\cap [0,y]) = y$. 
We multiply this inequality by $\chi_{E'}$ and integrate: 
$$\int_0^\omega \phi \left (\int_0^y \chi_{E'} dt \right ) \chi_{E'} 
dy \leq \int_0^\omega\phi(y) \chi_{E'}  dy = \int_{E'} \phi(y) dy = 
\int_E \phi(\Psi(x)) \psi(x) dx.$$
On the other hand, 
$$\int_0^\omega \phi \left ( \int_0^y \chi_{E'} dt \right ) 
\chi_{E'} dy = \Phi \left ( \int_{E'} dy \right ) = \Phi\left 
( \int_E \psi(x) dx \right ).$$
Putting these two inequalities together yields the 
inequality~(\ref{eqn_Szego}). 
 \end{proof}

\begin {lemma} \label{isop_lemma}
Let $G\subset \calW$ be a domain with compact closure. Then 
there is a function $\Upsilon_{\alpha}=\calF\circ Z^{-1}$ so that
\begin{equation}\label{Isop_eq}
\int_{\partial G} w^2\, ds \ge 
 \frac{\pi}{2\alpha}\,\Upsilon_{\alpha}
\left(\frac{2\alpha}{\pi}\int_{G} w^2\, da\right).
 \end{equation}
 Here $\calF(\rho)=\tan^{2\alpha}\!(\rho/2)\, \sin\rho$ and $Z$ is 
given by~(\ref{Z_eq}).
Equality holds if and only if $G$ is a sector $\calS(r)$. 
\end {lemma}
\begin{proof}
Map the domain $G$ into a domain $\tilde G$ in the upper halfplane using 
the transformation
$$
x= f(\rho)\,\cos\alpha\theta,\qquad y=f(\rho)\,\sin\alpha\theta,
$$
where we will choose $f$ to satisfy formula (\ref{def_f}). 
The Euclidean line element is
$$
dx^2 + dy^2 = \dot f^2 \, d\rho^2 + \alpha^2 f^2 \, d\theta^2.
$$
We claim that the map satisfies
\begin{equation}\label{moment_ineq}
\alpha^2 \tan^{4\alpha}\!\left(\frac {\rho}{2}\right) \, \sin^4 \alpha \theta
\, (d\rho^2 + \sin^2\!\rho\, d\theta^2)
\ge y^4(dx^2 + dy^2).
\end{equation}
For this to be true pointwise, we need the inequalities to hold
\begin{align}
\alpha \tan^{2\alpha}\!\left(\frac {\rho}{2}\right) &\ge
f^2\, \dot f=\left( \frac {f^3}3\right)'\label{desideratum1}\\
\sin \rho\, \tan^{2\alpha}\!\left(\frac {\rho}{2}\right)&\ge f^3.\label{desideratum2}
\end{align}
Expand $\sin \rho = 2\sin(\rho/2)\cos(\rho/2)$ and use equality in
inequality~(\ref{desideratum2}) to define $f$: 
\begin{equation}\label{def_f}
f=2^{\frac 13} \sin^{\frac{1+2\alpha }{3}} \left(\frac{\rho}2\right) 
\cos^{\frac{1-2\alpha}{3}} \left(\frac{\rho}2\right).
\end{equation}
Differentiating, we see 
$$
f^2\,\dot f = \tan^{2\alpha}\!\left(\frac {\rho}{2}\right)\left[ 
\frac{2\alpha + \cos \rho}3\right], 
$$
which implies that the  inequality~(\ref{desideratum1}) holds as well.

Equation~(\ref{w_eqn}) and  inequality~(\ref{moment_ineq}) imply that
\[ 
\alpha \int_{\partial G} w^2\, ds = \alpha \int_{\partial G}
w^2 \sqrt{d\rho^2 + \sin^2\!\rho\, d\theta^2}  \geq  \int_{\partial \tilde G} y^2 \sqrt{dx^2 + dy^2} := 
\calM (\partial \tilde G).
\]
The right side is the moment of inertia of a uniform mass distribution of 
the curve $\partial \tilde G$ relative to the $y$-axis. 
Among all domains with given fixed surface moment
$$
\int_{\tilde G} y^2\, dx\, dy,
$$
the semicircular arcs centered on the $y$-axis minimize 
$\calM(\partial \tilde G)$ \cite[Section 2]{PW}.
Compute $\calM(\partial \tilde G)$ and $\calM(\tilde G)$ in the case where 
$\partial \tilde G$ is a semicircle of radius $R$: 
$$\calM(\partial \tilde G) =\int_0^\pi R^3 \sin^2 t dt = \frac{\pi R^3}{2}, 
\qquad \calM(\tilde G) = \int_0^\pi \int_0^R r^3 \sin ^2 \theta  dr\, d\theta 
= \frac{\pi R^4}{8}.$$
Solving for $R$ in the formula for $\calM(\tilde G)$ above and 
using the fact that semicircles are minimizers, we see that for a 
general domain $\tilde G$ in the upper half plane
$$
\calM (\partial \tilde G) \geq 2^{\frac{5}{4}} \pi^{\frac{1}{4}} 
\left \{ \int_{\tilde G} y^2 dx \, dy \right \}^{\frac{3}{4}}.$$

 Returning to the original variables, $dx\, dy = \alpha f 
\dot f\, d\rho\, d\theta$ so
\begin {align} \label {AreaMomOrig_eqn}
\int_{\partial G} w^2 ds & \geq  \frac{1}{\alpha} 2^{\frac 54}\pi^{\frac 14}
\left \{ \int_{G}f^2\, \sin^2(\alpha \theta)\, \alpha f \dot f\, d\rho \,
d\theta \right \}^{\frac 34} \\ \nonumber 
& =  \left(\frac{\pi }{2\alpha}\right)^{\frac 14} \left\{
\int_{G}\frac 43 \left[
\tan^{2\alpha}\!\left(\frac {\rho}{2}\right)\,\sin\rho\right]^{\frac 13}\,
\left[ {2\alpha + \cos \rho}\right]\,
\tan^{2\alpha}\!\left(\frac {\rho}{2}\right)\, \sin^2\!\alpha\theta\, 
d\rho\, d\theta\right\}^{\frac 34}.
\end {align}
Choose $\beta$ so that
$$ \frac{2\alpha+2}{2\alpha+1}\le \beta<\frac 43.$$
 Regroup the integral inside the braces
$$
I=\frac {4}{3\beta} 
\int_{G} \left[
\tan^{2\alpha}\!\left(\frac {\rho}{2}\right)\,\sin\rho\right]^{\frac 43-\beta}\,
\left[ {2\alpha + \cos \rho}\right]\,
\beta\left[
\tan^{2\alpha}\!\left(\frac {\rho}{2}\right)\,\sin\rho\right]^{\beta-1}\,
\tan^{2\alpha}\!\left(\frac {\rho}{2}\right)\,  d\rho\, \sin^2\!\alpha\theta\, 
d\theta.
$$
Use Lemma~\ref{Szego_lemma}, with 
$$
\Psi=\left[\tan^{2\alpha}\left(\frac{\rho}{2}\right)\,\sin\rho\right]^{\beta} \Rightarrow 
\psi = \beta\left(\tan^{2\alpha}\left(\frac{\rho}{2}\right)\, \sin \rho\right)^{\beta - 1}
[2\alpha + \cos \rho]\, \tan^{2\alpha}\left(\frac{\rho}{2}\right)
$$ 
and 
$$
\phi(z)=\frac{4}{3\beta}z^{\frac 4{3\beta}-1} \Rightarrow  \Phi(z) =
z^{\frac{4}{3\beta}}.
$$ 
So that $\phi$ is increasing, we require $\beta< 
\frac 43$. If $H_{\theta}=\{ \rho \in [0,\pi):(\rho,\theta)\in G\}$ is the 
slice of $G$ in the $\rho$-direction then Szeg\H o's inequality~(\ref{eqn_Szego})  implies 
\begin{equation}\label{eqn_inside}
I\ge  \int_0^{\pi/\alpha}\left( \beta \int_{H_\theta} 
\tan^{2\alpha\beta}\!\left(\frac {\rho}{2}\right)\,\sin^{\beta-1}\!\rho\,
\left[ {2\alpha + \cos \rho}\right]\, d\rho
\right)^{\frac 4{3\beta}}\, \sin^2\!\alpha\theta\, d\theta.
\end{equation}
Equality holds if and only if $H_{\theta}=[0,r(\theta)]$ is an interval 
{\it a.e.}
Next we let $p = \frac{4}{3\beta} > 1$, $q = \frac{4}{4-3\beta}$,
and define the measure $d\nu = \sin^2 \alpha \theta\,  d\theta$. 
H\"older's inequality implies
\begin {gather*}
\left [ \int_0^{\pi/\alpha} \left ( \beta \int_{H_\theta} 
\tan^{2\alpha\beta}\left(\frac{\rho}{2}\right)\,\sin^{\beta -1}(\rho)\, [2\alpha + 
\cos \rho]\, d\rho \right )^p d\nu \right ]^{\frac 1p} \left [ 
\int_0^{\pi/\alpha} d\nu \right ]^{\frac 1q}  \\ 
\geq \int_0^{\pi/\alpha}
\beta \int_{H_\theta} \tan^{2\alpha \beta}\left(\frac{\rho}{2}\right)\,\sin^{\beta - 1}
(\rho)\, [2\alpha + \cos \rho]\, d\rho \, d\nu. 
\end {gather*}
Raising both sides of this inequality to the power $p$, 
rearranging, and using the fact that 
$$
\int_0^{\pi/\alpha} d\nu = \int_0^{\pi/\alpha} \sin^2 \alpha 
\theta \, d\theta = \frac{\pi}{2\alpha},
$$
(\ref{eqn_inside}) becomes
$$
I 
  \geq  \left( \frac{2\alpha}{\pi}\right)^{\frac 4{3\beta}-1}\,\left(
\beta \int_0^{\pi/\alpha}
\int_{H_\theta} 
\tan^{2\alpha\beta}\!\left(\frac {\rho}{2}\right)\,\sin^{\beta-1}\!\rho\,
\left[ {2\alpha + \cos \rho}\right]\, d\rho
\, \sin^2\!\alpha\theta\, d\theta \right)^{\frac 4{3\beta}}. 
$$

We regroup the inside integral again: 
$$J=
 \int_0^{\pi/\alpha}\int_{H_\theta}
\tan^{2\alpha(\beta-1)}\!\left(\frac{\rho}{2}\right)\,
\sin^{\beta-2}\!\rho\, \left[ {2\alpha + \cos \rho}\right]\,
\tan^{2\alpha}\!\left(\frac {\rho}{2}\right)\,\sin\rho\, d\rho\, 
\sin^2\!\alpha\theta\, d\theta.
$$
Let us denote 
\begin{equation}\label{Z_eq}
Z(r)=\int_0^r \tan^{2\alpha}\!\left(\frac {\rho}{2}\right)\,\sin\rho\, d\rho,
\end{equation}
 and define $\bar r(r,\theta)$ by
$$
Z(\bar r)=\int_0^r \tan^{2\alpha}\!\left(\frac {\rho}{2}\right)\,
\chi_{H_{\theta}}(\rho)\,\sin\rho\, d\rho,
$$
where $\chi_H$ denotes the characteristic function of $H$. The integrand 
$\tan^{2\alpha}(\rho/2) \sin \rho$ is positive and increasing for 
the range of $\rho$ we are considering, and so  
$\bar r(r,\theta)\le r$ with equality if and only if $H_{\theta}
\cap[0,r]=[0,r]$ {\it a.e.} If we require $(2\alpha + 1) \beta 
\geq 2\alpha + 2$, then the factor 
$$
g_\beta(\rho) = \tan^{2\alpha(\beta - 1)}\left(\frac{\rho}{2}\right)\, \sin^{\beta -2} 
\!\rho \, [2\alpha + \cos \rho]
$$
is increasing in $\rho$. Thus we can define $\Phi_\beta$ by 
\begin{equation}\label{Phi_eqn}
\phi_{\beta}(y)=\beta g_{\beta}\circ Z^{-1}(y), \qquad\qquad 
\Phi_{\beta}(y)=\int_0^y \phi_{\beta}(s)\, ds.
\end{equation}
Observe that $Z$ and $g_\beta$ are increasing, so $\phi_\beta$ is 
increasing and $\Phi_\beta$ is convex. Using $g_\beta(\bar r(\rho, \theta)) 
\leq g_\beta(\rho)$, we have 
\begin{align*}
J\ge&
\int_0^{\pi/\alpha}\int_{H_\theta} g_{\beta}(\bar r(\rho,\theta))\,
\tan^{2\alpha}\!\left(\frac {\rho}{2}\right)\,\sin\rho\, d\rho\, \sin^2\!
\alpha\theta\, d\theta\\
=&
 \frac 1{\beta}\int_0^{\pi/\alpha}\int_{H_\theta} \phi_\beta\left(
\int_0^{\rho} \tan^{2\alpha}\!\left(\frac {\rho'}{2}\right)\,
\chi_{H_{\theta}}(\rho')\,\sin\rho'\, d\rho'
\right)\,
\tan^{2\alpha}\!\left(\frac {\rho}{2}\right)\,\sin\rho\, d\rho\, 
\sin^2\!\alpha\theta\, d\theta.
\end{align*}

Now, using Lemma~\ref{Szego_lemma} with $\psi(\rho)=\tan^{2\alpha}
(\rho/2)\sin(\rho)\,\chi_{H_\theta}$ we have
$$
J\ge\frac 1{\beta}
\int_0^{\pi/\alpha}\Phi_{\beta}\left(
\int_{H_\theta}  \tan^{2\alpha}\!\left(\frac {\rho}{2}\right)\,\sin\rho\, d\rho
\right)
\, \sin^2\!\alpha\theta\, d\theta\
$$
with equality if and only if $H_{\theta}=[0,r(\theta)]$ is an interval 
{\it a.e.} Next, by Jensen's inequality (with the measure given by
$d\nu = \sin^2 \alpha \theta \,d\theta$),
$$
J\ge\frac{\pi}{2\alpha\beta}\Phi_{\beta}\left(\frac{2\alpha}{\pi}
\int_0^{\pi/\alpha}
\int_{H_\theta}  \tan^{2\alpha}\!\left(\frac {\rho}{2}\right)
\, \sin^2\!\alpha\theta \,\sin\rho\, d\rho
\, d\theta\right)
$$
with equality if and only if $\bar r(\theta)$ is {\it a.e.} constant.
Substituting back,
$$
I \ge \left ( \frac{2\alpha}{\pi} \right )^{\frac{4}{3\beta} -1}
(\beta J)^{\frac{4}{3\beta}} 
 \geq  \frac{\pi}{2\alpha}\,\left\{
 \Phi_{\beta}\left(\frac{2\alpha}{\pi}
\int_0^{\pi/\alpha}
\int_{H_\theta}  \tan^{2\alpha}\!\left(\frac {\rho}{2}\right)
\, \sin^2\!\alpha\theta \,\sin\rho\, d\rho
\, d\theta\right)\right\}^{\frac 4{3\beta}}.$$
Reinserting this back into~(\ref{AreaMomOrig_eqn}) yields
\begin {eqnarray} \label{Iso_eq}
\int_{\partial G} w^2\, ds & \ge & \left ( \frac{\pi}{2\alpha} 
\right )^{\frac 14}I^{\frac 34} \geq 
 \frac{\pi}{2\alpha}\,
 \Phi_{\beta}^{\frac 1{\beta}}
\left(\frac{2\alpha}{\pi}
\int_0^{\pi/\alpha}
\int_{H_\theta}  \tan^{2\alpha}\!\left(\frac {\rho}{2}\right)
\, \sin^2\!\alpha\theta \,\sin\rho\, d\rho
\, d\theta\right) \\ \nonumber 
& = & \frac{\pi}{2\alpha} \Phi_\beta^{\frac{1}{\beta}}
\left ( \frac{2\alpha}{\pi} \int_G w^2 da \right )
\end{eqnarray}
where equality holds if and only if also $\rho(\theta)$ is constant {\it a.e.}
Notice that the right hand side of this inequality is always bounded 
by $\int_{\partial G}w^2 ds$, and so we can use the Dominated Convergence 
Theorem to take a limit as $\beta \rightarrow \frac{4}{3}$ from below. In 
other words, (\ref{Iso_eq}) holds for $\beta=\frac 43$.

Let us compute $\Phi_{\beta}^{\frac{1}{\beta}}\left(Y\right)$. Since it depends 
only on~(\ref{Phi_eqn}), it would be the same for any function  $v^*$ whose 
level sets $G^*_{\eta}=\{ x: v^*(x)\ge \eta\}$ give the same value for the 
integral of $w^2$ (see~(\ref{zeta_eqn}) below). In this case, we choose a 
spherical rearrangement whose levels are the sectors $G^*_\eta = \calS(r(\eta))$.
Expressing things in terms of $r(\eta)$, we have 
\begin{equation}\label{ChV_eq}
\frac{2\alpha}{\pi}y=\frac{2\alpha}{\pi}\zeta(\eta)=\frac{2\alpha}{\pi}
\int_{\calS\bigl(r(\eta)\bigr)} w^2\, da = Z\bigl(r(\eta)\bigr)
\end{equation}
so, changing variables $s=Z(r)$
\begin{align}
\Phi_{\beta}\left(Y\right)&=\int _0^Y \phi_{\beta}(s)\, ds\notag \\
&= \beta\int_0^{Z^{-1}(Y)} g_{\beta}(r) \, \tan^{2\alpha}\!\left(
\frac r2\right)\sin r\, dr\notag\\
&=\beta\int_0^{Z^{-1}(Y)}\left[
\tan^{2\alpha}\!\left(\frac{r}{2}\right){\sin r}\right]^{\beta-1} 
\left[ {2\alpha + \cos r}\right]\, \, \tan^{2\alpha}\!\left(\frac r2\right)\, 
dr\notag\\
&= \left[
\tan^{2\alpha}\!\left(\frac{Z^{-1}(Y)}{2}\right){\sin (Z^{-1}(Y))}\right]^{\beta}
.\label{PhiBeta_eq}
\end{align}
Observe that we get the same equation~(\ref{Iso_eq}) for all $\beta$.
Thus we set $\Upsilon_{\alpha}= \Phi_{\beta}^{\frac 1{\beta}}$ 
in~(\ref{Iso_eq}) giving~(\ref{Isop_eq}).
\end{proof}

It is precisely at inequality~(\ref{desideratum1}) where the analagous 
proof in the hyperbolic case 
fails. In the hyperbolic case, the harmonic weight function is 
$w(\rho,\theta) = \tanh^{2\alpha}(\rho/2) \sin (\alpha\theta)$, 
and versions of equations (\ref{moment_ineq}), (\ref{def_f}) hold 
with $\cos$ replaced by $\cosh$ and $\sin$ replaced by $\sinh$. 
This choice of $f$ gives us
$$f^2\,\dot f = \tanh^{2\alpha}\!\left(\frac {\rho}{2}\right)\left[ 
\frac{2\alpha + \cosh \rho}3\right], $$
much like the formula above, but this does not yield $f^2 \,\dot f
\leq \alpha \tanh^{2\alpha}(\rho/2)$, because $\cosh\rho$ grows 
exponentially with $\rho$. To rememdy this problem, one can 
try to vary the power of $\sinh(\rho/2)$ or $\cosh(\rho/2)$; however
this will only yield a worse inequality for $f^2 \, \dot f$.

\section{Estimate of Rayleigh Quotient.} \label {rayliegh_sec}
 
Theorem \ref{Main_th} now follows as in \cite{PW}. Let  $G\subset 
\Ss^2$  be a domain that lies in the wedge $\calW = \{ (\rho,\theta): 0\le \rho,
\ 0\le \theta\le \pi/\alpha\}$. It suffices to  estimate the Rayleigh quotient 
for admissible functions $u\in C^2_0(G)$ that  are twice continuously 
differentiable and compactly supported in $G$. Any admissible function may 
be written $u=v w$ using the harmonic function (\ref{w_eqn}) and $v\in 
C^2_0(G)$. The divergence theorem shows
$$
\int_G |du|^2\, da = \int_G w^2\, |dv|^2 \, da.
$$
Let  $G_{t}$ denote the points of $G$ satisfying  $v\ge t$. 
Putting
\begin{align}
\zeta(t) &= \int_{G_t} w^2 \, da,\label{zeta_eqn}
\end{align}
we see that  $\zeta(0)=\hat\zeta\ge \zeta(t)\ge 0=\zeta(\hat v)$, where $\hat v 
= \max_Gv$,
$$
\frac{\partial \zeta}{\partial t} = - \int_{\partial G_t} \dfrac{w^2}
{|dv|}\, ds
$$
and 
$$
\int_G w^2\, v^2\, da = \int_0^{\hat v} 2t\,\zeta(t)\, dt=\int_0^{\hat 
\zeta} t^2 d\zeta.
$$

Then,  using   the coarea formula, Schwarz's inequality, 
Lemma~\ref{isop_lemma}, and changing variables to $y=\zeta(t)$, 
the inequality~(\ref{Isop_eq}) implies
\begin {eqnarray} \label {1st-rayleigh-est}
\int_G w^2\,|d v|^2\, da 
& \ge & \int_0^{\hat v}\left\{ \int_{\partial G_t} w^2\, |dv|
\, ds\right\} \, dt \\ \nonumber
& \ge & \int_0^{\hat v}\frac{\left\{ \int_{\partial G_t} w^2\, 
ds\right\}^2}{\int_{\partial G_t} \dfrac{w^2}{|dv|}\, ds } \, 
dt\\ \nonumber
& \ge & \frac{\pi^2}{4\alpha^2}\,
\int_0^{\hat v}\frac{\Upsilon_{\alpha}^2\left(\dfrac{2\alpha}{\pi}
\zeta(t)\right)}{-\dfrac{\partial \zeta}{\partial t}} \, dt.
\end{eqnarray}

Changing variables to $y=\zeta(t)$ we have
\begin{equation} \label {2nd-rayleigh-est}
\int_0^{\hat \zeta}\Upsilon_{\alpha}^2\left(\frac{2\alpha}{\pi}y\right)
\left(\frac{\partial t}{\partial y}\right)^2 \, dy\ge \mu \int_0^{\hat \zeta}
t(y)^2\, dy
\end{equation}
where $\mu$ is the least eigenvalue of the boundary value problem
\begin{gather}
\frac{\partial}{\partial y}\left(\Upsilon_{\alpha}^2\left(\frac{2\alpha}
{\pi}y\right)\frac{\partial q}{\partial y}\right)+\mu\,  q = 0,\label{OneD_eq}\\ 
q(\hat\zeta)=0,\qquad\qquad \lim_{y\to 0+} \Upsilon_{\alpha}^2\left(\frac{2\alpha}
{\pi}y\right)\frac{\partial q}{\partial y}=0.\label{OneDBC_eq}
\end{gather}

Now perform the change variables in~(\ref{OneD_eq}) and~(\ref{OneDBC_eq}) given 
by~(\ref{ChV_eq}), so that the domain is now $[0, r^*]$, $Z(r^*) = \frac{2\alpha}
{\pi} \hat \zeta$, and $\mu$ is now the least eigenvalue of
\begin{gather}
\frac{\partial}{\partial r}\left(\tan^{2\alpha}\!\left(\frac r2\right)\sin(r)\, 
\frac{\partial q}{\partial r}\right)+\frac{\pi^2\mu}{4\alpha^2}\tan^{2\alpha}\!
\left(\frac r2\right)\sin (r) q = 0,\label{ROneD_eq}\\
q(r^*)=0,\qquad\qquad \lim_{r\to 0+} \tan^{2\alpha}\!\left(\frac r2\right)\sin (r)
\,\frac{\partial q}{\partial r}=0.\label{ROneDBC_eq}
\end{gather}
Note that~(\ref{ROneD_eq}) is the eigenequation for the spherical sector 
$\calS(r^*)$. Hence $\frac{\pi^2\mu}{4\alpha^2}=\lambda_1(\calS(r^*))$.

Reassembling using equations (\ref{1st-rayleigh-est}) and 
(\ref{2nd-rayleigh-est}), we get the inequality
$$
\int_G |d u|^2\, da 
 \ge  \lambda_1\bigl(\calS(r^*)\bigr)\, \int_G u^2\, da, 
$$
which implies the inequality~(\ref{EV_eq}). 

\section{Computation of the lower bound and applications.} \label {appl_sec}

The eigenvalue $\lambda^*=\lambda_1(\calS(r^*))$ occurs as the eigenvalue of 
the problem~(\ref{ROneD_eq}), (\ref{ROneDBC_eq}) on $[0,r^*]$, which may be 
rewritten
\begin{gather*}
\sin(r)\, q'' + [2\alpha + \cos(r)]\, q' + \lambda^*\sin(r)\, q = 0;\\
\lim_{r\to 0-}\tan^{2\alpha}\!\left(\frac r2\right)\,\sin(r)\,
\frac{dq}{dr}(r)=0,\qquad\qquad q(r^*)=0.
\end{gather*}
Making the change of variable $x=\frac 12({1-\cos r})$ transforms the ODE 
to the hypergeometric equation on $[0,1]$
\begin{gather*}
x(1-x)\, \ddot y+[c-(a+b+1)x]\, \dot y - ab\, y  =  0,\\
\lim_{x\to 0-}x^{\alpha+1}\,
\frac{dy}{dr}(x)=0,\qquad\qquad q(x^*)=0.
\end{gather*}
with
$$
a,b = \frac{1\pm\sqrt{1+4\lambda^*}}{2},\qquad c={\alpha+1}.
$$
The solution to the hypergeometric equation is Gau\ss's ordinary 
hypergeometric function, given by
$$
{}_2\text{F}_1(a ,b ; c ; x )= 1+\frac{a b }{c }\frac x {1!}+
\frac{a (a +1)b (b +1)}{c (c +1)}\frac
{x ^2}{2!}+\frac{a (a +1)(a +2)b (b +1)(b +2)}{c (c +1)(c +2)}
\frac {x ^3}{3!}+\cdots .
$$
We find the eigenvalue by a shooting method. Given $r^*$, 
$\lambda^*$ is the first positive root of the function
\begin{equation}\label{RootEV_eq}
\lambda\mapsto {}_2\text{F}_1\left( \frac{1-\sqrt{1+4\lambda}}{2} , 
\frac{1+\sqrt{1+4\lambda}}{2} ; {\alpha+1} ; \frac{1-\cos r^*}2 \right).
\end{equation}

\begin{table}\label{EV_tab}
\begin{center}
\begin{tabular}{| c | l l l l |}
\hline
$G$ & $\calI(G)$ & $r^*$ & $\lambda_1(G)$ & $\lambda_1(\calS(r^*))$\\ 
\hline\hline
$\calW$ & $\infty$  & $\pi$  &  $(\alpha+1)\alpha$ & $(\alpha+1)\alpha$ \\
$\calS(\frac{\pi}2)$ &  $\frac{\pi}{2\alpha}Z\left(\frac{\pi}2\right)$   
& $\frac{\pi}2$  & $(\alpha+1)(\alpha+2)$  & $(\alpha+1)(\alpha+2)$ \\
$\calS(r)$ & $\frac{\pi}{2\alpha}Z(r)$   &  $r$ & $\lambda^*$  & $\lambda^*$ \\ 
\hline
$\calW$, $\alpha=\frac 32$ & $\infty$   & 3.14159265  & 3.75  & 3.75  \\
$\calS\left(\delta\right)$, $\alpha=\frac 32$ &  2.07876577  &  2.18627604 
& 5.00463538  &  5.00463538\\
$\calS(\varepsilon)$, $\alpha=\frac 32$ &   0.90871989  & 1.91063324  
& 6.19561775  &  6.19561775\\
$\calS(\frac {\pi}2)$, $\alpha=\frac 32$ &    0.30118555  & 1.57079633  
& 8.75  & 8.75 \\
$\calT$ &    1.88896324 &  2.15399460 & 5.1590\ldots        &  5.11641465\\
$\hat\calT$ &  1.90831355 & 2.15742981 & ? & 5.10421518\\
\hline
\end{tabular}
\caption{Domains and eigenvalues. In this table $\delta = \cos^{-1}
(-1/\sqrt{3})$ and $\varepsilon = \cos^{-1}(-1/3)$. Values not described  
are taken from~\cite{RT}.
}
\end{center}
\end{table}

Consider the example of the geodesic triangle $\calT\subset\Ss^2$ which is 
a face of the regular tetrahedral tessellation, whose vertices in the unit 
sphere could be taken as $\left(\frac 1{\sqrt 3},\pm \sqrt{\frac 23},0\right)$ 
and $\left(-\frac 1{\sqrt 3},0,\pm\sqrt{\frac 23}\right)$. The distance 
between vertices is $\varepsilon=\cos^{-1}\left(-\frac 13 \right)$. The 
diameter, which equals the distance from vertex to center of the opposite 
edge, is $\delta=\cos^{-1}\left(-\frac 1{\sqrt 3}\right)$. $\calT$ fits 
inside a wedge sharing a vertex of angle $\frac{2\pi}3$.  Writing
$$
\calT = \left\{ (\rho,\theta):0\le\theta\le \frac{2\pi}3,\quad 0\le \rho\le 
r(\theta)\right\}
$$
we find
$$
r(\theta)= \frac{\pi}2+\Atn\left(\frac{ \cos(\theta-\frac{\pi}3)}{\sqrt 2}
\right).
$$
At the vertex we have $\alpha=\frac 32$ so that
$$
Z(r)=\int_0^r \tan^3\!\left(\frac {\rho}2\right)\,\sin\rho\, d\rho=
4\tan\!\left(\frac {r}2\right)+\sin r-3r.
$$
$\lambda_1(\calT)$ was found numerically in~\cite{RT}. Using the 
computer algebra system  {\sc Maple}$\copyright$, we numerically integrate
$$
\calI(\calT)=\int_0^{\pi/\alpha} Z(r(\theta))\, \sin^2(\alpha\theta)\, d\theta
$$
and solve $\frac{\pi}{2\alpha}Z(r^*)=\calI(\calT)$ for $r^*$ 
and~(\ref{RootEV_eq}) for $\lambda^*$   to get the other values in the 
$\calT$ line in Table~1. 

To avoid the quadrature, we observe the estimate
$$
Z(r(\theta))\le T(\theta):= A_1 + A_2\cos\Bigl( \theta-\frac{\pi}3\Bigr)
+A_3\Bigl( 1-\cos(6\theta)\Bigr),
$$
where $A_1$ and $A_2$  are chosen so that the functions agree at $\theta=0$ 
and   $\theta=\frac{\pi}3$ and the $A_3$ is chosen to make the second 
derivatives agree at $\frac{\pi}3$. The inequality follows since the second 
derivative of the difference goes from negative to positive  in $0< 
\theta<\pi/3$.
 This corresponds to the larger domain $\hat \calT$ whose radius function 
is $\hat r(\theta)= Z^{-1}(T(\theta))$. 
Then
\begin{equation}\label{HatT_eq}
\frac{\pi}{2\alpha}Z(\hat r^*)=\int_{\hat\calT} w^2\, da = 
\int_0^{\frac{2\pi}3} T(\theta)\, \sin^2\!\left(\frac 32\theta\right)\, 
d\theta=\frac{\pi}3A_1+\frac{9\sqrt 3}{16}A_2+\frac{\pi}3A_3.
\end{equation}
Using these values we obtain the last row of Table~1. By eigenvalue 
monotonicity, if $\hat\calT\supset\calT$ then $\lambda_1(\calT)\ge
\lambda_1(\hat\calT)$. 

This eigenvalue estimate provides an alternative to our argument~\cite{RT} 
in a Brownian pursuit problem.  We finished the missing ($n=4$) case in a 
proof by  Li and Shao~\cite{LS} of the conjecture of Bramson and 
Griffeath~\cite{BG}.
\begin{cor} Suppose the prey $X_0(t)$ is chased by $n$ pursuers $X_1(t),
\ldots,X_n(t)$, all doing independent standard Brownian motions on the 
line. Suppose that the pursuers start to the left of the prey $X_j(0)
<X_0(0)$ for all $j=1,\ldots,n$. Then the expected capture time is finite 
if and only if $n\ge 4$. 
\end{cor}
In fact, for the  capture time for $n$ pursuers
$$
\tau_n=\inf\{ t>0: X_j(t)\ge X_0(t)\text{ for some $j\ge 1$}\}
$$
there are finite constants~$a(n)$, and~$C$ depending on the initial 
position and the eigenvalue of the link of the pursuit cone\cite{DB} 
so that the probability
$$
{\mathbb P}(\tau_n>t) \sim C\, t^{-a}\qquad\text{as $t\to\infty$}.
$$ 
The proof shows $a(n)>1$ and thus $\mathbb E\tau_n<\infty$ if and only if $n\ge 4$. Our eigenvalue estimates 
give the following  corresponding bounds on the decay rates since they are 
related by a formula to  the eigenvalue estimates~\cite{RT}. From the 
estimate on $\hat \calT$, $a(3)\ge .90695886$ and so $a(4)\ge 1.00029446$; 
from the estimate of $\calT$ involving quadrature, $a(3)\ge .90827616 $ 
and $a(4)\ge  1.00151234$.

\begin{proof} Details are provided in \cite{RT}. Finiteness of the 
expectation  of $\tau_4$ follows if it can be shown that $\lambda_1(\calT)
>5.101267527$. 
The lower eigenvalue bound is given by Theorem~\ref{Main_th} applied to 
$\calT$  depends on either the numerical integration of $\calI(\calT)$ or 
its upper bound by the quadrature free estimate of~(\ref{HatT_eq}).
\end{proof}
 
\medskip
\small{\noindent  {\bf Acknowledgement.} The second author thanks the 
University of California, Irvine, for  his visit while this work was 
completed. We also thank Lofti Hermi for bringing~\cite{PW} to our attention.}


\begin {thebibliography}{999}


\bibitem [BG]{BG} M. Bramson \& D. Griffeath. {\em Capture problems
for coupled random walks}.
Random Walks, Brownian Motion and Interacting Particle Systems
(R. Durrett \& H. Kesten, ed.).
  Birkh\"auser, 1991.
  
\bibitem[Ch]{Ch} I. Chavel. Eigenvalues in Riemannian Geometry. 
Series in Pure and Applied Mathematics. 115. Academic 
Press, Inc., Orlando, 1984. 

\bibitem [C]{C} J. Cheeger. {\em A lower bound for the 
smallest eigenvalue of of the Laplacian.} in Problems in Analysis.
Princeton University Press, 1970. 

\bibitem [DB]{DB} R. D. DeBlassie. {\em Exit times from cones in $\R^n$ of
Brownian motion}. Prob.
Theory and Rel. Fields. {\bf 74} (1987), 1--29.

\bibitem [F]{F} C. Faber. {\em Beweiss, dass unter allen homogenen 
Membrane von gleicher Fl\"ache und gleicher 
Spannung die kreisf\"ormige die tiefsten Grundton gibt.} 
Sitzungsber.--Bayer. Akad. Wiss., Math.--Phys. Munich. (1923), 169--172. 

\bibitem [K]{K} E. Krahn. {\em \"Uber eine von Rayleigh formulierte 
Minmaleigenschaft des Kreises.} Math. Ann. {\bf 94} (1925), 97--100. 

\bibitem [LT]{LT} P. Li \& A. Treibergs. {\em Applications of 
eigenvalue techniques to geometry}. Contemporary Geometry: 
J.-Q.~Zhong Memorial Volume
(H.-H. Wu, ed.). University Series in Mathematics (Plenum Press),  
New York, 1991., pp. 22--54. 

\bibitem [LS]{LS} W. Li \& Q.-M. Shao. {\em Capture time of Brownian
pursuits}. Prob. Theory and Rel. Fields. {\bf 121} (2001), 30--48.

\bibitem[P]{P} L. Payne. {\em Isoperimetric inequalities and their 
applications}. SIAM Review. {\bf 9} (1967), 453--488.

\bibitem[PW]{PW} L. Payne \& H. Weinberger. {\em A Faber-Krahn inequality 
for wedge-like membranes.} Journal of Mathematics and Physics. {\bf 39} 
(1960) 182--188.

\bibitem [RT]{RT} J. Ratzkin \& A. Treibergs. {\em A capture problem in 
Brownian motion and eigenvalues of spherical domains.}
to appear in Transaction of the Amer. Mathematical Soc. (2007).

\bibitem [Sz]{Sz} G. Szeg\H o. {\em \" Uber eine Verallgemeinerung des 
Dirichletschen Integrals.} Math. Zeit. {\bf 52} (1950),  676--685.

\end {thebibliography}
\end{document}